\documentclass{article}

\usepackage{amsmath}
\usepackage{amsthm}
\usepackage{amssymb}

\newcommand\ND{\newcommand}

\ND\lref[1]{Lemma~\ref{#1}}
\ND\tref[1]{Theorem~\ref{#1}}
\ND\pref[1]{Proposition~\ref{#1}}
\ND\sref[1]{Section~\ref{#1}}
\ND\rref[1]{Remark~\ref{#1}} 
\ND\corref[1]{Corollary~\ref{#1}}
\ND\eref[1]{Example~\ref{#1}}

\ND\map[3]{#1\!:\!#2\!\to\!#3}
\ND\ts{\!\times\!}

\ND\PD[2]{\frac{\partial#1}{\partial#2}}
\ND\half{\frac{1}{2}}

\ND\R{\mathbb{R}}
\ND\N{\mathbb{N}}
\ND\Q{\mathbb{Q}}
\ND\Z{\mathbb{Z}} 
\ND\X{\mathbb{X}}
\ND\e{\epsilon} 
\ND\ve{\varepsilon} 
\ND\Rd{\R ^d }

\ND\Ae{\mathsf{A}_{\e }^{\mathbf{a} } (\7,\6,\8)}
\ND\Aa{\mathsf{A}_{\e}^- }
\ND\A{\mathsf{A}_{\e} }
\ND\Ab{\mathsf{A}_{\e}^+ }

\ND\Aee{\Ab \backslash \Aa }

\ND\0{\mathrm{Cap}} 
\ND\1{I _{\7 }}  
\ND\Ia{I_{a_\7}} 
\ND\2{I }

\ND\3{ a_{\7} + \frac{1}{\8} }\ND\4{ h _{\e }   } \ND\5{ \mathfrak{h} _{\e }   } 
\ND\6{ n } \ND\7{ r }\ND\8{ m }

\ND\za{\Theta ^{\mathbf{a} }{(\7,\6)}}
\ND\zb{\Theta ^{\mathbf{a} }{(\7,\6,\8)}}
\ND\zc{\Theta ^{\mathbf{a} }(\7 + \frac{1}{\8}, \6 )}

\ND\di{\mathcal{D} _{\infty}}  

\ND\J{\mathfrak{J} _{r,\e  /2 }}

\ND\Lm{L^2(\Theta,\mu)}

\ND\yaQ{\1^{\6}} \ND\ff{\mathfrak{f} } \ND\ya{\Theta^n_r} \ND\frn{f_{\7}^{\6}}
\ND\dii{\mathcal{D} _{\infty }^{loc}}

\ND\xir{\mathbf{x}^{\6}_{\7}(\theta )} 
\ND\ir{^{\6}_{\7,\theta }} \ND\Ri{\R^{(\6)}}

\ND\dom{\mathcal{D} } \ND\ed{(\mathcal{E}, \di )} \ND\ee{(\mathcal{E}, \dom )}
\ND\diffusion{(\{ \mathsf{P}_{\theta} \}_{\theta \in \Theta }, \{ \mathbb{X}_t \} ) }

\ND\Ereg{\mathcal{E}^{\text{reg}} }
\ND\Dreg{\dom ^{\text{reg}}}

\begin{document}

\begin{center}{\bf 
Non-collision and collision properties of \\
Dyson's model in infinite dimension and other stochastic dynamics 
 whose equilibrium states are determinantal random point fields
}
\end{center}

\bigskip
\begin{center}{\sf 
Hirofumi Osada}\\
Dedicated to Professor Tokuzo Shiga on his 60th birthday 
\end{center}

\medskip
\begin{center}Published in \\
Advanced Studies in Pure Mathematics 39, 2004\\
Stochastic Analysis on Large Scale Interacting Systems\\
pp.~325--343
\end{center}

\bigskip

\begin{abstract}
Dyson's model on interacting Brownian particles is a stochastic 
dynamics consisting of 
an infinite amount of particles moving in $ \R $ 
with a logarithmic pair interaction potential.  
For this model we will prove that each pair of particles never collide.  

The equilibrium state of this dynamics is a determinantal 
random point field with the sine kernel. 
We prove for stochastic dynamics given 
by Dirichlet forms with determinantal random point fields 
as equilibrium states  the particles never collide 
if the kernel of determining random point fields 
are locally Lipschitz continuous, and give examples of collision 
when H\"{o}lder continuous. 

In addition we construct infinite volume dynamics 
(a kind of infinite dimensional diffusions) whose equilibrium states 
are determinantal random point fields. 
The last result is partial in the sense that we simply construct 
a diffusion associated with the {\em maximal closable part} of 
{\em canonical} pre Dirichlet forms 
for given determinantal random point fields as equilibrium states. 
To prove the closability of canonical pre Dirichlet forms 
for given determinantal random point fields is still an open problem. 
We prove these dynamics are the strong resolvent limit of 
finite volume dynamics. 
\end{abstract}


\numberwithin{equation}{section}
\newcounter{Const} \setcounter{Const}{0}

	\def\Ct{\refstepcounter{Const}c_{\theConst}}	
	\def\cref#1{c_{\ref{#1}}}	

\theoremstyle{theorem}
\newtheorem{thm}{Theorem}[section]  
\theoremstyle{plain}
\newtheorem{lem}{Lemma}[section]
\newtheorem{cor}[lem]{Corollary}
\newtheorem{prop}[lem]{Proposition}
\theoremstyle{definition}
\newtheorem{dfn}[lem]{Definition}
\newtheorem{ntt}[lem]{Notation}
\newtheorem{cdn}[lem]{Condition}
\newtheorem{ex}[lem]{Example}
\newtheorem{prob}[lem]{Problem}
\theoremstyle{remark}
\newtheorem{rem}[lem]{Remark}

\section{Introduction} \label{s:1} 
Dyson's model on interacting Brownian particles  in infinite dimension 
is an infinitely dimensional diffusion process $ \{ (X_t ^{i} )_{i \in \N}\}  $ 
formally given by the following stochastic differential equation (SDE):
\begin{align}\label{:11}&
d X_t^i = dB_t^i + \sum_{j=1,\ j\not= i} ^{\infty}\frac{1}{X_t^i - X_t^j} 
\, dt \quad (i = 1,2,3,\ldots)
				,			\end{align}
where $ \{ B_t ^{i} \}  $ is an infinite amount of independent one dimensional Brownian motions. 
The corresponding unlabeled dynamics is 
\begin{align}\label{:12}&
\mathbb{X}_t = \sum _{i=1}^{\infty } \delta _{X_t^i}
.							\end{align}
Here $ \delta _{\cdot } $ denote the point mass at $ \cdot $.  
By definition $ \mathbb{X}_t $ is a $ \Theta  $-valued diffusion, where 
$ \Theta  $ is the set consisting of configurations on $ \R $; that is, 
\begin{align}&\label{:13}
\Theta  = \{ \theta = \sum _{i} \delta _{x_i}\, ;\,  x_i \in \R,\, \theta (\{ | x | \le r\}) < \infty  
\text{ for all } r \in \R\} 
.\end{align}  
We regard $ \Theta  $ as a complete, separable metric space with the vague topology.

In \cite{sp.2} Spohn constructed an unlabeled dynamics \eqref{:12} in the sense of 
a Markovian semigroup on $ \Lm $. 
Here $ \mu  $ is a probability measure on $ (\Theta  , \mathfrak{B}(\Theta  ) ) $ 
whose correlation functions are generated by the sine kernel 
\begin{align}\label{:14}&
\mathsf{K} _{\sin} (x)= \frac{\bar{\rho}}{\pi x } \sin (\pi x ).
\end{align}
(See \sref{s:2}). 
Here $ 0 < \bar{\rho} \le 1 $ is a constant related to the {\em density} of the particle. Spohn indeed proved the closability of a non-negative bilinear form 
$ (\mathcal{E}, \di ) $ on $ \Lm $ 
\begin{align} \label{:15} &
\mathcal{E}(\mathfrak{f},\mathfrak{g}) = \int_\Theta \mathbb{D}[\mathfrak{f},\mathfrak{g}](\theta )d\mu, 
\\ & \notag 
\di=\{\ff \in \dii \cap \Lm ; \mathcal{E}(\mathfrak{f},\mathfrak{f})
 <\infty \}.
							\end{align}
Here $ \mathbb{D} $ is the square field given by \eqref{:25} and 
$ \dii $ is the set of the local smooth functions on $ \Theta $ 
(see \sref{s:3} for the definition). The Markovian semi-group is given by the Dirichlet form that is the closure $ (\mathcal{E},\dom ) $ 
of this closable form on $ \Lm $.

The measure $ \mu $ is an equilibrium state of 
\eqref{:12}, whose formal Hamiltonian 
$ \mathcal{H} =  \mathcal{H} (\theta )  $ 
is given by $ ( \theta = \sum _{i} \delta _{x_i}) $ 
\begin{align}& \label{:16}
\mathcal{H}(\theta ) = \sum_{i\not=j} - 2 \log |x_i - x_j|  
,\end{align}
which is a reason we regard Spohn's Markovian semi-group is a 
correspondent to the dynamics formally given by 
the SDE \eqref{:11} and \eqref{:12}.  

We remark the existence of an $ L^2 $-Markovian semigroup does not imply the existence of 
the associated {\em diffusion} in general. 
Here a diffusion means (a family of distributions of) a strong Markov process with 
continuous sample paths starting from each $ \theta \in \Theta $. 

In \cite{o.dfa} it was proved that there exists a diffusion 
$ \diffusion $ 
with state space $ \Theta $ 
associated with the Markovian semigroup above. 
This construction admits us to investigate the {\em trajectory-wise} 
properties of the dynamics. 
In the present paper we concentrate on the collision property of the diffusion. 
The problem we are interested in is the following:

\vspace{1ex}

{\sf Does a pair of particles $ (X_t^i,X_t^j) $ that collides 
each other for some time $ 0 < t < \infty $ exist ?}
  
\vspace{1ex}

We say for a diffusion on $ \Theta $ {\em the non-collision occurs}
 if the above property does {\em not} hold, and {\em the collision occurs} if otherwise. 

If the number of particles is finite, then the non-collision should occur at least intuitive level. 
This is because drifts $ \frac{1}{x_i-x_j} $ have a strong repulsive effect. 
When the number of the particles is infinite, the non-collision property is 
non-trivial because the interaction potential is 
long range and un-integrable. 
We will prove the non-collision property holds 
for Dyson's model in infinite dimension.

Since the sine kernel measure is the prototype of 
determinantal random point fields, 
it is natural to ask such a non-collision property 
is universal for stochastic dynamics given by Dirichlet forms \eqref{:15} 
 with the replacement of the measure $ \mu $ with general determinantal random point fields. We will prove, if the kernel of the determinantal random point field 
(see \eqref{:23}) is 
locally Lipschitz continuous, then the non-collision always occurs. 
In addition, we give an example of determinantal random point fields with 
H\"{o}lder continuous kernel that the collision occurs.

The second problem we are interested in this paper is the following: 

\vspace{1ex}

\textsf{Does there exist $ \Theta $-valued diffusions associated with the Dirichlet forms $ (\mathcal{E},\dom ) $ on $ \Lm $ when 
$ \mu $ is determinantal random point fields ?}

\vspace{1ex}

We give a partial answer for this in \tref{l:28}.

\vspace{1ex}

The organization of the paper is as follows: 
In \sref{s:2} we state main theorems. 
In \sref{s:3} we prepare some notion on configuration spaces. 
In \sref{s:4} we prove \tref{l:23} and \tref{l:25}. 
In \sref{s:5} we prove \pref{l:27} and \tref{l:26}. 
In \sref{s:6} we prove \tref{l:28}. 
Our method proving \tref{l:20} can be applied to Gibbs measures. 
So we prove the non-collision property for Gibbs measures in \sref{s:7}.

\section{Set up and the main result} \label{s:2}

Let $ \mathsf{E} \subset \Rd $ be a closed set which is the closure 
of a connected open set in $ \Rd $ with smooth boundary.  
Although we will mainly treat the case $ \mathsf{E}= \R  $, 
we give a general framework here by following the line of \cite{so-}.  
Let $ \Theta $ denote the set of configurations on $ \mathsf{E}  $, 
which is defined similarly as \eqref{:13} by replacing 
$ \R $ with $ \mathsf{E}  $.  

A probability measure on $ (\Theta ,\mathcal{B}(\Theta )  ) $ is called 
a random point field on $ \mathsf{E}  $. 
Let $ \mu $ be a random point field on $ \mathsf{E}  $. 
A non-negative, permutation invariant function 
$ \map{\rho _n}{\mathsf{E}^n }{\R} $ is called an $ n $-correlation function 
of $ \mu $ 
if for any measurable sets $ \{ A_1,\ldots,A_m \} $ 
and natural numbers $ \{ k_1,\ldots,k_m \}  $ such that 
$ k_1+\cdots + k_m = n $ the following holds:
\begin{align*}&
\int _{A_1^{k_1}\ts \cdots \ts A_{m}^{k_m}} 
\rho _n (x_1,\ldots,x_n) dx_1\cdots dx_n = 
\int _{\Theta } \prod _{i=1}^m 
\frac{\theta (A_i) ! }{(\theta (A_i) - k_i ) ! }
d\mu 
.\end{align*}
It is known (\cite{so-}, \cite{le-1}, \cite{le-2}) that, if a family of 
non-negative, permutation invariant functions $ \{ \rho _n \}  $ satisfies 
\begin{align}\label{:21}&
\sum_{k=1}^{\infty} 
\left\{ \frac{1}{(k+j)!}  \int _{A^{k+j}} \rho _{k+j} \ 
dx_1\cdots dx_{k+j} 
\right\} ^{-1/k} = \infty
, \end{align}
then there exists a unique probability measure (random point field) $ \mu $ 
on $ \mathsf{E}  $ whose correlation functions equal $\{ \rho _n \}  $. 

Let $ \map{K}{L^2(\mathsf{E}, dx )}{L^2(\mathsf{E}, dx )} $ be 
a non-negative definite operator which is locally trace class; namely  
\begin{align}\label{:22}&
0 \le (K f , f ) _{L^2(\mathsf{E}, dx )},\quad 
\\ \notag &
\mathrm{Tr} (1_B K 1_B ) < \infty \quad \text{ for all bounded Borel set }
B 
.\end{align}
We assume $ K $ has a continuous kernel denoted by 
$ \mathsf{K} = \mathsf{K} (x,y) $. 
Without this assumption one can develop a theory of determinantal random point fields 
(see \cite{so-}, \cite{shirai-takahashi}); we assume this 
for the sake of simplicity. 

\begin{dfn}\label{dfn:1}
A probability measure $ \mu $ on $ \Theta $ is said to be a determinantal 
(or fermion) random point field with kernel $ \mathsf{K} $ 
if its correlation functions $ \rho _n $ are given by 
\begin{align}\label{:23}&
\rho _n (x_1,\ldots,x_n) = \det (\mathsf{K} (x_i,x_j)_{1\le i,j \le n })
\end{align}
\end{dfn} 

We quote: 
\begin{lem}[Theorem 3 in \cite{so-}]   \label{l:21}
Assume $ \mathsf{K}(x,y) = \overline{\mathsf{K}(y,x) }  $ and $ 0 \le K \le 1 $. 
Then $ K $ determines a unique determinantal random point field $ \mu $.  

\end{lem}
We give examples of determinantal random point fields. 
The first example is the stationary measure of Dyson's model in infinite dimension. 
The first three examples are related to the semicircle law of empirical distribution of eigen values of random matrices. We refer to \cite{so-} for detail. 
\begin{ex}[sine kernel]\label{d:21} 
Let $ \mathsf{K} _{\sin} $ and $ \bar{\rho} $ be as in \eqref{:14}. 
Then 
\begin{align}\label{:2m1}&
\mathsf{K} _{\sin} (t) = \frac{1}{2\pi } \int _{|k| \le \pi \bar{\rho} } e ^{\sqrt{-1} k t} \, dk 
.							\end{align}
So the $ \mathsf{K} _{\sin}  $ is a function of positive type and 
satisfies the assumptions in \lref{l:21}. 
Let $ \hat{\mu }^{ N } $ denote the probability measure on $ \R^{ N } $ defined by 
\begin{align}\label{:2m4}&
 \hat{\mu }^N = \frac{1}{Z^{N}} 
e ^{ -  \sum _{i, j =1 }^{ N } - 2 \log |x_i-x_j| }
e^{ - \lambda ^2 _{N}\sum_{i=1}^{ N } x _i ^2 } 
  \, dx_1 \cdots dx_{ N }
	,						\end{align}
where $ \lambda _{N} = 2 (\pi \bar{\rho })^3 / 3 N ^2 $ and $ Z^{N} $ is the normalization. 
Set $ \mu ^{ N } = \hat{\mu }^{ N } \circ (\xi ^{ N })^{-1}$, where 
$ \map{\xi ^{ N }}{\R^{ N }}{\Theta } $ such that 
$ \xi ^{ N }(x_1,\ldots,x_N) = \sum _{i=1}^{ N } \delta _{x_i}$.   
Let $  \rho _{ n } ^{ N } $ denote the $ n $-correlation function of $ \mu ^{ N } $.  
Let $ \rho _n $ denote the $ n $-correlation function of $ \mu $. 
Then it is known (\cite[Proposition 1]{sp.2},  \cite{so-}) that 
for all $ n = 1,2,\ldots  $ 
\begin{align}\label{:2m5}&
\lim_{ N \to\infty}  \rho _{ n } ^{ N } (x_1,\ldots,x_n) = \rho _n (x_1,\ldots,x_n) 
\quad \text{ for all }(x_1,\ldots,x_n ) 
.							\end{align}
In this sense the measure $ \mu $ is associated 
with the Hamiltonian $ \mathcal{H}  $ 
in \eqref{:16} coming from the log potential $ -2\log |x| $.  

\end{ex}
\begin{ex}[Airy kernel]\label{d:22}
$ \mathsf{E}=\R  $ and  
$$ \mathsf{K}(x,y)= \frac{\mathcal{A}_i(x) \cdot \mathcal{A}_i'(y)
- \mathcal{A}_i(y) \cdot \mathcal{A}_i'(x)
}{x-y}  
$$
Here $ \mathcal{A}_i $ is the Airy function.   
\end{ex}
\begin{ex}[Bessel kernel]\label{d:23}
Let $ \mathsf{E} = [0, \infty)  $ and 
\begin{align*}&
\mathsf{K} (x,y) = \frac{
J_{\alpha}(\sqrt{x})\cdot \sqrt{y}\cdot J_{\alpha}'(\sqrt{y})
-
J_{\alpha}(\sqrt{y})\cdot \sqrt{x}\cdot J_{\alpha}'(\sqrt{x})
}
{2(x-y)} 
.\end{align*}
Here $ J_{\alpha } $ is the Bessel function of order $ \alpha $.
 \end{ex}

\begin{ex}\label{d:24} %
Let $ \mathsf{E} = \R  $ and 
$ \mathsf{K}(x,y) = \mathsf{m}(x) \mathsf{k}(x-y) \mathsf{m}(y)   $, where  
$ \map{\mathsf{k}}{\R }{\R } $ is a non-negative, continuous {\em even} function that is convex in $ [0, \infty) $ such that $ \mathsf{k} (0) \le 1 $, 
and $ \map{\mathsf{m}}{\R }{\R }   $ is nonnegative continuous and $ \int _{\R } \mathsf{m}(t) dt < \infty  $ and $ \mathsf{m} (x) \le 1  $ for all $ x  $ and 
$ 0 < \mathsf{m} (x) $ for some $ x $.  
Then $ \mathsf{K}  $ satisfies the assumptions in \lref{l:21}. 
Indeed, it is well-known that $ \mathsf{k}  $ is a function of positive type 
(187 p. in \cite{don} for example), so the Fourier transformation of a finite positive measure. By assumption $ 0 \le \mathsf{K}(x,y) \le 1 $,
which  implies $ 0 \le K \le 1 $. Since $ \int \mathsf{K} (x,x) dx < \infty $, 
$ K $ is of trace class. 
\end{ex}


Let $ \mathsf{A}$ denote the subset of $ \Theta  $ defined by
\begin{align}\label{:24}&
\mathsf{A}= \{ \theta \in \Theta ;\ 
\theta (\{ x \} ) \ge 2 
\quad \text{ for some } x \in \mathsf{E} \}  
					.		\end{align}
Note that $ \mathsf{A}$ denotes the set consisting of the configurations with collisions. 
We are interested in how large the set $ \mathsf{A} $ is. 
Of course $ \mu ( \mathsf{A}) = 0 $ because the 2-correlation function is locally 
integrable. We study $ \mathsf{A}$ more closely from the point of stochastic dynamics; 
namely, we measure $ \mathsf{A}$ by using a capacity.

To introduce the capacity we next consider  a bilinear form related to 
the given probability measure $ \mu $. 
Let $\dii $ be the set of all local, smooth functions on $\Theta $ 
defined in \sref{s:3}. 
For $\ff ,\mathfrak{g} \in\dii $ we set 
$\map{\mathbb{D}[\mathfrak{f},\mathfrak{g}]}{\Theta}{\R}$ by 
\begin{align} \label{:25}
\mathbb{D}[\mathfrak{f},\mathfrak{g}](\theta )&= \frac{1}{2}\sum_{i} 
\frac{\partial f (\mathbf{x} )}{\partial x_{i}}\frac{\partial g (\mathbf{x} )}{\partial x_{i}} 
.\end{align}
Here $ \theta = \sum _{i} \delta _{x_i}$, $ \mathbf{x} = (x_1,\ldots ) $   and 
$ f (\mathbf{x} )= f(x_1,\ldots )$ is the permutation invariant function such that 
$\ff (\theta )=f (x_1,x_2,\ldots )$ for all $\theta \in \Theta $. 
We set $ g $ similarly. Note that the left hand side of \eqref{:25} is again 
permutation invariant. Hence it can be regard as a function of $  \theta = \sum _{i} \delta _{x_i}$.  
Such  $ f $ and $ g $ are unique; so the function $\map{\mathbb{D}[\mathfrak{f},\mathfrak{g}]}{\Theta }{\R }$  is well defined. 

For a probability measure $ \mu $ in $ \Theta $ we set as before 
\begin{align} \notag &
\mathcal{E}(\mathfrak{f},\mathfrak{g}) = \int_\Theta \mathbb{D}[\mathfrak{f},\mathfrak{g}](\theta )d\mu, 
\\ & \notag 
\di=\{\ff \in \dii \cap \Lm ; \mathcal{E}(\mathfrak{f},\mathfrak{f})
 <\infty \}.
							\end{align}
When $ (\mathcal{E}, \di ) $ is closable on $  \Lm $, we denote its closure by 
$ (\mathcal{E}, \dom ) $. 

We are now ready to introduce a notion of capacity for a {\em pre}-Dirichlet space 
$ (\mathcal{E}, \di , \Lm ) $.  
Let $ \mathcal{O}  $ denote the set consisting of all open sets in $ \Theta $. 
For $ O \in \mathcal{O}  $ we set 
$ \mathcal{L}_{O} =\{ \mathfrak{f}  \in \di \, ;\, \mathfrak{f}  \ge 1 \ 
\mu \text{-a.e.\ on } {O} \}  $ and   
\begin{align*}&
\0 ({O}) = 
\begin{cases}
\inf_{\mathfrak{f} \in \mathcal{L}_{O} } \left\{ \mathcal{E}(\mathfrak{f}, \mathfrak{f}  ) 
+ (\mathfrak{f}, \mathfrak{f} )_{\Lm } \right\} &\mathcal{L}_{O}\not= \emptyset
\\
\infty 
&\mathcal{L}_{O}=\emptyset 
\end{cases}
.
\end{align*}
For an arbitrary subset $ {A} \subset \Theta $ we set 
$ \0 ({A}) = \inf _{{A} \subset {O} \in \mathcal{O} } \0 ({O})$. 
This quantity $ \0 $ is called 1-capacity for the pre-Dirichlet space 
$ (\mathcal{E}, \di , \Lm ) $.  

We state the main theorem:
\begin{thm}\label{l:20}
Let $ \mu $ be a determinantal random point field with kernel $ \mathsf{K}  $. 
Assume $ \mathsf{K}  $ is locally Lipschitz continuous. 
Then 
\begin{align}\label{:26}&
\0 ( \mathsf{A})=0
,							\end{align}
where $ \mathsf{A}  $ is given by \eqref{:24}. \end{thm}

In \cite{o.dfa} it was proved
\begin{lem}[Corollary 1 in \cite{o.dfa}]\label{l:22}
Let $ \mu $ be a probability measure on $ \Theta $. 
Assume $ \mu $  has locally bounded correlation functions. 
Assume $ (\mathcal{E}, \di ) $ is closable on $ \Lm $. 
Then there exists a diffusion $ \diffusion $ 
associated with the Dirichlet space $ (\mathcal{E}, \dom , \Lm ) $.  
\end{lem}

Combining this with \tref{l:20} we have 
\begin{thm}	\label{l:23}
Assume $ \mu $ satisfies the assumption in \tref{l:20}. 
Assume $ (\mathcal{E}, \di ) $ is closable on $ \Lm  $. 
Then a diffusion $ \diffusion $ associated with the Dirichlet space 
$ (\mathcal{E}, \dom , \Lm )  $ exists and satisfies 
\begin{align}\label{:27}&
\mathsf{P}_{\theta}( \sigma _{\mathsf{A}} = \infty ) = 1 \quad \text{ for q.e.\ }\theta 
,							\end{align}
where $ \sigma _{\mathsf{A}} = \inf \{ t > 0\, ;\, \X _t \in \mathsf{A}\} $.  
\end{thm}
We refer to \cite{fot} for q.e.\ (quasi everywhere) and related notions on Dirichlet form theory. We remark the capacity of pre-Dirichlet forms are bigger than or 
equal to the one of its closure by definition. So 
\eqref{:27} is an immediate consequence of \tref{l:20} and 
the general theory of Dirichlet forms 
once $ (\mathcal{E}, \di ) $ is closable on $ \Lm $ and the resulting 
(quasi) regular Dirichlet space $ (\mathcal{E}, \dom , \Lm ) $ exists.

To apply \tref{l:23} to Dyson's model we recall a result of Spohn. 
\begin{lem}[Proposition 4 in \cite{sp.2}]\label{l:24}
Let $ \mu $ be the determinantal random point field 
with the sine kernel in \eref{d:21}. 
Then $ \ed $ is closable on $ \Lm $.  
\end{lem}

We say a diffusion $ \diffusion $ is Dyson's model in infinite dimension 
if it is associated with 
the Dirichlet space $ (\mathcal{E}, \dom , \Lm  ) $ in \tref{l:24}. 
Collecting these we conclude: 
\begin{thm}	\label{l:25}
No collision \eqref{:27} occurs in Dyson's model in infinite dimension. 
\end{thm}

The assumption of the local Lipschitz continuity of the kernel $ \mathsf{K}  $ 
is crucial; 
we next give a collision example when $ \mathsf{K}  $ is 
merely H\"{o}lder continuous. 
We prepare: 
\begin{prop}   \label{l:27}
Assume $ K $ is of trace class. 
Then $ (\mathcal{E}, \di ) $ is closable on $ \Lm $. 
\end{prop}

\begin{thm}	\label{l:26}
Let $ \mathsf{K} (x,y) = \mathsf{m}(x) \mathsf{k}(x-y) \mathsf{m}(y)  $ 
be as in \eref{d:24}. 
Let $ \alpha $ be a constant such that 
\begin{align}\label{:27a}&
0 < \alpha < 1 
.\end{align}
Assume $ \mathsf{m}  $ and $ \mathsf{k}  $ are continuous and 
there exist positive constants 
$ \Ct \label{;21} $ and $ \Ct \label{;22} $ such that 
\begin{align}\label{:27b}&
\cref{;21}t ^{\alpha} 
\le \mathsf{k} (0) - \mathsf{k}(t) \le 
\cref{;22}t ^{\alpha} 
\quad \text{ for } 0 \le t \le 1.
\end{align}
Then $ (\mathcal{E}, \di , \Lm ) $ is closable and the associated diffusion 
satisfies 
\begin{align}\label{:27c}&
\mathsf{P}_{\theta}(\sigma _{\mathsf{A}} < \infty ) = 1 \quad \text{ for q.e.\ } \theta 
.\end{align}
\end{thm}

Unfortunately the closability of the pre-Dirichlet form 
$ (\mathcal{E} , \di ) $ on $ \Lm $ has not yet proved for 
determinantal random point fields of locally trace class except the sine kernel. 
So we propose a problem: 
\begin{prob}\label{p:22}
\thetag{1} Are pre-Dirichlet forms $ (\mathcal{E} , \di ) $ on $ \Lm $ 
closable when $ \mu  $ are determinantal random fields with continuous kernels? 
\\
\thetag{2} 
Can one construct stochastic dynamics (diffusion processes) associated with 
pre-Dirichlet forms $ (\mathcal{E} , \di ) $ on $ \Lm $. 
\end{prob}
We remark one can deduce the second problem from the first one (see \cite[Theorem 1]{o.dfa}). 
We conjecture that $ (\mathcal{E} , \di , \Lm ) $ are always closable. 
As we see above, 
in case of trace class kernel, this problem is solved by \pref{l:27}. 
But it is important to prove this for determinantal random point field of 
{\em locally} trace class. 
This class contains Airy kernel and Bessel kernel and other nutritious examples. 
We also remark for interacting Brownian motions with Gibbsian equilibriums 
this problem was settled successfully (\cite{o.dfa}). 

In the next theorem we give a partial answer for \thetag{2} of Problem~\ref{p:22}. 
We will show one can construct a stochastic dynamics in infinite volume, 
which is canonical in the sense that \thetag{1} it is the strong resolvent limit of 
a sequence of finite volume dynamics and that \thetag{2} it coincides with 
$ (\mathcal{E}, \dom ) $ whenever $ (\mathcal{E}, \di ) $ is closable 
on $ \Lm $.  

For two symmetric, nonnegative forms $ (\mathcal{E}_1,\dom _1 ) $ and 
$ (\mathcal{E}_2 ,\dom _2 ) $, we write 
$ (\mathcal{E}_1,\dom _1 ) \le (\mathcal{E}_2,\dom _2 ) $ if 
$ \dom _1 \supset \dom _2 $ and 
$ 
\mathcal{E}_1 (\mathfrak{f}, \mathfrak{f} ) \le 
\mathcal{E}_2 (\mathfrak{f},\mathfrak{f} )$ 
for all $ \mathfrak{f} \in \dom _2 $. 
 Let $ (\mathcal{E}^{\text{reg}} , \dom ^{\text{reg}}) $ denote the regular part of  $ (\mathcal{E} , \di ) $ on $ \Lm $, that is, 
$ (\Ereg , \Dreg ) $ is closable on $ \Lm $ and in addition satisfies the following:  
\begin{align*}&
(\mathcal{E}^{\text{reg}} , \dom ^{\text{reg}}) \le (\mathcal{E} ,\di  )
,\\
\intertext{and for all closable forms such that 
$ (\mathcal{E}', \mathcal{D}'  ) \le (\mathcal{E} ,\di  )  $}
 &
(\mathcal{E}', \mathcal{D}'  ) \le (\Ereg , \Dreg ) 
.\end{align*} 
It is well known that such a $ (\Ereg , \Dreg ) $ exists uniquely and called 
the maximal regular part of $ (\mathcal{E}, \dom ) $. 
Let us denote the closure by the same 
symbol $ (\Ereg , \Dreg ) $. 

Let $ \map{\pi _r}{\Theta }{\Theta } $ be such that 
$ \pi _r (\theta ) = \theta (\cdot \cap \{ x \in \mathsf{E} ; |x| < r \} )$. 
We set 
\begin{align*}&
\dom _{\infty , r} = \{ \mathfrak{f} \in \di \, ;\, \mathfrak{f}  \text{ is }
\sigma[\pi _r] \text{-measurable}\} 
.\end{align*}
We will prove $ (\mathcal{E}, \dom _{\infty , r}  ) $ are closable on 
$ \Lm $. These are the finite volume dynamics we are considering.  

Let $ \mathbb{G}_{\alpha } $ 
(resp.\ $ \mathbb{G}_{r , \alpha } $) $ (\alpha >0) $ denote the 
$ \alpha $-resolvent of the semi-group associated with the closure of 
$ (\Ereg ,\Dreg ) $ (resp.\ $ (\mathcal{E}, \dom _{\infty , r}  $)) on 
$ \Lm $.   
\begin{thm}	\label{l:28} 
\thetag{1} 
$ (\Ereg , \Dreg ) $ on $ \Lm $ is a quasi-regular Dirichlet form. 
So the associated diffusion exists. 
\\
\thetag{2} $ \mathbb{G}_{r , \alpha } $ converge to $ \mathbb{G}_{\alpha} $  
strongly in $ \Lm $ for all $ \alpha >0 $.
\end{thm}

\begin{rem}\label{r:222} 
We think the diffusion constructed in \tref{l:28} is a reasonable one because 
of the following reason. 
\thetag{1} By definition the closure of $ (\Ereg , \Dreg ) $ equals 
$ (\mathcal{E}, \dom  ) $ when $ (\mathcal{E}, \di ) $ is closable. 
\thetag{2} One naturally associated Markov processes on 
$ \Theta _r $, where $ \Theta _r $ is the set of configurations on 
$ \mathsf{E} \cap \{ |x| < r \}   $.  
So \thetag{2} of \tref{l:28} implies the diffusion is the strong resolvent limit of finite volume dynamics. 
\end{rem}
\begin{rem}\label{r:223}
If one replace $ \mu $ by the Poisson random measure $ \lambda $ 
whose intensity measure is the Lebesgue measure and consider 
the Dirichlet space $ (\mathcal{E}^{\lambda},\dom ) $ 
on $ L(\Theta , \lambda ) $, then the associated $ \Theta $-valued 
diffusion is the $ \Theta $-valued Brownian motion $ \mathbb{B} $, 
that is, it is given by 
\begin{align*}&
\mathbb{B}_t = \sum _{i=1}^{\infty} \delta _{B^i_t} 
,\end{align*}
where $ \{ B^i_t \} $ ($ i \in \N $ ) are infinite amount of independent 
Brownian motions. 
In this sense we say in Abstract that the Dirichlet form given by \eqref{:15} 
for Radon measures in $ \Theta $ {\em canonical}.  
We also remark such a type of local Dirichlet forms are often called 
distorted Brownian motions. 
\end{rem}

\section{Preliminary} \label{s:3}
Let $ \1 = (-\7,\7)^d \cap \mathsf{E} $ and 
$ \ya =\{\theta \in\Theta ;\theta(\1 )= \6 \} $. 
We note $\Theta =\sum _{\6 = 0} ^{\infty}\ya $. 
Let $\yaQ  $ be the $ \6 $ times product of $ \1 $.  
We define $ \map{\pi_r}{\Theta  }{\Theta  } $ by $\pi _{\7}(\theta )=\theta(\cdot\cap \1 )$.  
A function $\map{\mathbf{x}}{\ya }{\yaQ } $ is called a $\yaQ  $-coordinate of $\theta $ if 
\begin{align}\label{:31}&
\pi _{\7}(\theta )=\sum _{k=1} ^{\6 } \delta_{x_{k}(\theta )}, 
\qquad \mathbf{x}(\theta )=(x_{1}(\theta ),\ldots,x_{\6 }(\theta )).
							\end{align}
Suppose $ \map{\ff }{\Theta }{\R }$ is $ \sigma [\pi _{\7}]$-measurable. 
Then for each $ n = 1,2,\ldots $ there exists a unique permutation invariant function 
$ \map{\frn }{\yaQ }{\R } $ such that 
\begin{align}\label{:32}&
 \ff (\theta ) = \frn (\mathbf{x}(\theta ) ) \quad \text{ for all } \theta \in \ya 
.							\end{align}

We next introduce mollifier. 
Let $\map{j}{\R }{\R } $ be a non-negative, smooth function such that 
$ j (x)= j (|x|)$, 
$\int_{\R^d} j dx=1$ and $ j (x)=0$ for $ |x| \geq \frac{1}{2}$. 
Let $ j _\e =\e  j (\cdot/\e )$ and 
$ j ^{\6}_{\e } (x_1,\ldots,x_{\6})=\prod _{i=1}^{\6}  j _\e(x_i)$.   
For a $ \sigma [\pi _{\7}]$-measurable function $\ff $ we set 
$\map{\mathfrak{J}  _{\7,\e }\ff }{\Theta }{\R }$ by 
\begin{align} \label{:33} 
\mathfrak{J}  _{\7,\e }\ff (\theta ) &= 
\begin{cases}
j _{\e }^{\6} * \hat{f}^{\6}_r (\mathbf{x}(\theta )) & \text{ for } \theta \in \ya , \6 \ge 1 
\\
\ff (\theta ) & \text{ for } \theta \in \Theta ^0_r, 
\end{cases}
\end{align}
where $ \frn  $ is given by \eqref{:32} for $\ff $, and 
$\map{\hat{f}^{\6}_{\7}}{\R^{d\6}}{\R}$ is the function defined by 
$\hat{f}^{\6} _{\7}(x)=f^{\6} _{\7}(x)$ for $ x \in \yaQ $ and 
$\hat{f}^{\6} _{\7}(x)=0$ for $ x \not\in \yaQ $.  
Moreover $\mathbf{x}(\theta ) $ is an $ \1^{\6}  $-coordinate of $\theta \in \ya $, 
and $\ast $ denotes the convolution in $ \R ^{\6} $. 
It is clear that $ \mathfrak{J}  _{\7,\e }\ff $ is $ \sigma [\pi _{\7}]$-measurable. 

We say a function $ \map{\ff }{\Theta }{\R } $ is local if $ \ff $ is 
$ \sigma [\pi _{\7}]$-measurable for some $ \7 < \infty $. 
For $ \map{\ff }{\Theta }{\R } $ and $ \6 \in \N \cup \{ \infty  \}  $ 
there exists a unique permutation function 
$ f^{\6} $ such that $ \ff (\theta ) = f^{\6}(x_1,\ldots) $ 
for all $ \theta \in \Theta ^{\6 } $. 
Here 
$ \Theta ^{\6} = \{ \theta \in \Theta \, ;\, \theta (\mathsf{E}  ) = \6 \}  $, 
and $  \theta = \sum _{i} \delta _{x_i}$. 
A function $ \ff $ is called smooth if $ f^{\6 } $ is smooth for all $ \6 \in \N \cup \{ \infty  \}  $. Note that a $ \sigma [\pi _{\7}]$-measurable function $ \ff $ is smooth if and only if 
$ \frn $ is smooth for all $ n \in \N $.

\section{Proof of \tref{l:23}} \label{s:4}
We give a sequence of reductions of \eqref{:26}. Let $ \mathbf{A}  $ denote the set 
consisting of the sequences $ \mathbf{a} = (a_r)_{r \in \N} $ satisfying the following: 
\begin{align}\label{:41a}&
a_r \in \Q \quad \text{ for all }r \in \N ,
\\ & \label{:41b}
a_r = 2r + r_0   \quad \text{ for all sufficiently large }r \in \N
,\\ & \label{:41c}
2  \le  a_1 ,\ 1 \le a_{r+1} - a_r \le 2 \quad \text{ for all }r \in \N 
. 							\end{align}
Note that the cardinality of $ \mathbf{A}  $ is countable by \eqref{:41a} and \eqref{:41b}. 

Let  $ \mathbb{I}= \{ 2,3,\ldots, \}^3 $.    
For $ (\7,\6,\8) \in \mathbb{I}$ and $ \mathbf{a} = (a_{\7}) \in \mathbf{A}  $ we set  
\begin{align*}&
 \za = \{ \theta \in \Theta  \, ;\, \theta ( \Ia ) = \6 \}  
 \\&
  \zb = \{ \theta \in \Theta  \, ;\, \theta ( \Ia ) = \6 ,\ 
  \theta ( \bar {\2} _{a_{\7} + \frac{1}{\8} } \backslash \Ia  ) = 0 \} 
.\end{align*}
Here 
$ \bar {\2} _{a_{\7} + \frac{1}{\8} } $  
is the closure of $  {\2} _{a_{\7} + \frac{1}{\8} } $, 
where $ \1 = (-\7,\7)^d \cap \mathsf{E} $ as before. 
We remark $ \zb $ is an open set in $ \Theta  $.   
We set 
\begin{align}& \label{:42}
 \Ae = \{   \theta = \sum _{i} \delta _{x_i}  \, ;\, \theta  \in \zb \text{ and $ \theta  $ satisfy } 
\\ &\quad \quad \notag 
| x_i - x_j | < \e   \text{ and } x_i, x_j \in \2 _{a_{\7} - 1}  \text{ for some }i \not= j \} 
.\end{align}
It is clear that $ \Ae  $ is an open set in $ \Theta  $. 
\begin{lem}   \label{l:43} 
Assume that for all 
$ \mathbf{a} \in \mathbf{A}  $ and $  (\7,\6,\8) \in \mathbb{I}$ 
\begin{align}\label{:43}&
\inf _{ 0 < \e  < 1/2m} \0 ( \Ae ) = 0
.							\end{align}
Then \eqref{:26} holds. 
\end{lem}
\begin{proof}
Let 
\begin{align}\notag  
\mathsf{A}^{\mathbf{a} } (\7,\6,\8)= &\{   \theta = \sum _{i} \delta _{x_i}  \, ;\, \theta \in \zb \text{ and $ \theta  $ satisfy } 
\\ \quad \quad   &\notag 
\ x_i = x_j  \text{ and } x_i, x_j \in \2 _{a_{\7} - 1}  \text{ for some }i \not= j \}
.\end{align}
Then 
$ \mathsf{A}= \bigcup _{\mathbf{a}\in \mathbf{A}  }\bigcup _{ (\7,\6,\8) \in \mathbb{I}} \, \mathsf{A}^{\mathbf{a} } (\7,\6,\8) 
 $. 
Since $ \mathbf{A}  $ and $ \mathbb{I}$ are countable sets and the capacity is 
sub additive, \eqref{:26} follows from 
\begin{align}\label{:44}&
\0 (\mathsf{A}^{\mathbf{a} } (\7,\6,\8) )= 0 \quad \text{ for all }  
\mathbf{a} \in \mathbf{A}, \ (\7,\6,\8) \in \mathbb{I} 
.							\end{align}
Note that 
$  \mathsf{A}^{\mathbf{a} } (\7,\6,\8) \subset \Ae  $. 
So \eqref{:43} implies \eqref{:44} by the monotonicity of the capacity, 
which deduces \eqref{:26}. 
\end{proof}

Now fix $ \mathbf{a} \in \mathbf{A}  $ and $  (\7,\6,\8) \in \mathbb{I}$ 
and suppress them from the notion. Set 
\begin{align}\label{:45}&
\Aa = \mathsf{A}_{\e /2 }^{\mathbf{a} } (\7,\6,\8) 
,\quad 
\A = \Ae 
,\quad 
\Ab = \mathsf{A}_{1 + \e  }^{\mathbf{a} } (\7,\6,\8) 
.\end{align}
and let $ \map{\4 }{\R}{\R} $ ($ 0 < \e < 1/m < 1 $) such that 
\begin{align}\label{:46}&
\4 (t) = 
\begin{cases} 
2 								& ( |t| \le \e )
\\  2 \log |t| / \log \e  		&( \e \le |t| \le 1  ) 
\\0 							& ( 1  \le |t| )
.\end{cases}
\end{align}
We define $ \map{\5}{\Theta  }{\R } $ by $ \5 (\theta )= 0 $ for $ \theta \not\in \zb $ and 
\begin{align*}&
\5 (\theta ) = 
\sum _{ x _i , \,  x _j \in \2 _{a_{\7} - 1}, \  j\ne i} 
\4 (x _i - x _j )
\quad \text{ for }\theta \in \zb 
.\end{align*}
Here we set $ \5 (\theta ) = 0 $ if the summand is empty. 
Let $ \mathfrak{g}_{\e  } = \mathfrak{J}_{a_{\7} + \frac{1}{m}, \e / 4} \5 $. 
Here $ \mathfrak{J}_{a_{\7} + \frac{1}{m}, \e / 4}  $ is 
the mollifier introduced in  \eqref{:33}. 

\begin{lem}   \label{l:44}
For $ 0 < \e  < 1/2m  $,  $ \mathfrak{g}_{\e  } $ satisfy the following:
\begin{align}
\label{:47}&
\mathfrak{g}_{\e  } \in \di 
&& 
\\  & \label{:48b}
\mathfrak{g}_{\e  } (\theta ) \ge 1 
&&
\text{ for all } \theta \in \A  
\\ & \label{:48a}
0 \le \mathfrak{g}_{\e  } (\theta ) \le  n(n+1) 
&& 
\text{ for all } \theta \in \Theta 
\\ & \label{:48c}
 \mathfrak{g}_{\e  } (\theta ) =0  
&& \text{ for all } \theta \not\in \Ab   
\\ & \label{:49b}
\mathbb{D}[\mathfrak{g}_{\e  } , \mathfrak{g}_{\e  } ] (\theta ) = 0 
&& 
\text{ for all } \theta \not\in \Aee 
\\ \label{:49a}&
\mathbb{D}[\mathfrak{g}_{\e  } , \mathfrak{g}_{\e  } ] (\theta ) \le 
\frac{ \cref{;41} }{( \log \e \  \min  |x_i - x_j | ) ^2 } 
&& \text{ for all } \theta \in  \Aee 
.				\end{align}
Here $ \theta = \sum \delta _{x_k} $ and the minimum in \eqref{:49a} is taken over 
$ x_i , x_j  $ such that 
$$  x _i , \,  x _j \in \2 _{a_{\7} - 1}, \ \e / 2 \le |x_i - x_j | \le 1 + \e 
,$$ 
and 
$\Ct \label{;41} \ge 0 $ is a constant independent of $ \e  $ 
($ \cref{;41} $ depends on $  (\7,\6,\8)  $ ). 
\end{lem}
\begin{proof}
\eqref{:47} follows from \cite[Lemma 2.4 (1)]{o.dfa}. 
Other statements are clear from  a direct calculation. 
\end{proof}

Permutation invariant functions $\map{\sigma^{\6 }_{\7 }}{\yaQ }{\R^+}$ are called 
density functions of $\mu $ if, for all bounded $\sigma [\pi_r]$-measurable functions $ \ff $, 
\begin{align}\label{:3.5}&
\int_{\ya } \ff \, d\mu  = 
\frac{1}{\6 !} \int_{\yaQ }f^{\6 }_{\7 }\sigma^{\6 }_{\7 }dx 
.							\end{align}
Here $\map{f^{\6 }_{\7 }}{\yaQ }{\R}$ is the permutation 
invariant function such that $ f^{\6 }_{\7 }(\mathbf{x}(\theta ))=\ff (\theta ) $ 
for $\theta\in\ya $, where $\mathbf{x}$ is an $\yaQ $-coordinate. 
We recall relations between a correlation function and a density function 
(\cite{so-}): 
\begin{align}\label{:4t}&  
\rho _{\6} = \sum_{k=0}^{\infty} \frac{1}{k !}
\int _{I_r^{k }} \sigma ^{\6 + k}_{\7} (x_1,\ldots , x_{\6+k}) 
dx_{\6+1}\cdots dx_{\6+k}
\\ \label{:4u} &
\sigma ^{\6 }_{\7} = \sum_{k=0}^{\infty} 
\frac{(-1)^k}{k !}
\int _{I_r^{k }} \rho _{\6 + k} (x_1,\ldots , x_{\6+k}) 
dx_{\6+1}\cdots dx_{\6+k}
\end{align}
The first summand in the right hand side of \eqref{:4t} is taken to be 
$ \sigma ^{\6 }_{\7 } $. It is clear that 
\begin{align}\label{:4v}&
0 \le \sigma ^{\6 }_{\7 } ( x_1,\ldots,x_n ) 
\le \rho _{\6 } ( x_1,\ldots,x_n )
\end{align}

\begin{lem}   \label{l:42}
There exists a constant $ \Ct \label{;4x} $ depending on $ \7, \6 $ such that 
\begin{align}\label{:4x}&
\sigma^{\6 }_{\7 } (x_1,\ldots,x_n) \le 
\cref{;4x} 
\min _{ i \not= j} |x_i - x_j |
\quad \text{ for all } (x_1,\ldots,x_n) \in \yaQ 
\end{align}
\end{lem}
\begin{proof}
By \eqref{:23} and the kernel $ \mathsf{K}  $ is locally Lipschitz continuous, 
we see $ \rho _{\6} $ is bounded  and Lipschitz continuous on $ \yaQ $. 
In addition, by using \eqref{:23} we see $ \rho _{\6} = 0 $ if 
$ x_i=x_j $ for some $ i \not= j $.  
Hence by using \eqref{:23} again 
there exists a constant $ \Ct \label{;4y} $ depending on $ \6 , \7  $ such that 
\begin{align}\label{:4y}&
\rho _{\6} (x_1,\ldots,x_n) \le \cref{;4y} \min _{ i \not= j} |x_i - x_j | 
\quad \text{ for all } (x_1,\ldots,x_n) \in \yaQ 
.\end{align}
\eqref{:4x} follows from this and \eqref{:4v} immediately. 
\end{proof}

\begin{lem}   \label{l:45}
\eqref{:43} holds true. 
\end{lem}
\begin{proof}
By the definition of the capacity, $ \mathfrak{g}_{\e  } \in \di $, 
\eqref{:47} and \eqref{:48b} we obtain 
\begin{align}\label{:4a}&
\0 (\A ) \le \mathcal{E} (\mathfrak{g}_{\e  } , \mathfrak{g}_{\e  } ) + (\mathfrak{g}_{\e  } ,\mathfrak{g}_{\e  } )_{\Lm } 
					\end{align} 
So we will estimate the right hand side. We now see by \eqref{:49b}
\begin{align}
\label{:4b} 
\mathcal{E}(\mathfrak{g}_{\e  } , \mathfrak{g}_{\e  } ) & 
= \int _{ \Aee  } \mathbb{D}[\mathfrak{g}_{\e  } , \mathfrak{g}_{\e  } ] d \mu 
\\& \notag 
= \frac{1}{n !} \int _{ B _{\e }} 
\{ \frac{1}{2}\sum_{i=1} ^{\6 }
\frac{\partial g _{\e }^{\6}}{\partial x_{i}}
\frac{\partial g _{\e }^{\6}}{\partial x_{i}} 
\} \sigma _{\3}^{\6}  dx_1 \cdots dx_n 
\\ \notag  & 
= : \mathbf{I}_{\e } 
.\end{align}
Here $ g _{\e }^{\6} $ is defined by \eqref{:32} for $ \mathfrak{g}_{\e  } $, 
and  
$ B _{\e }= \varpi _{\3} ^{-1} ( \pi _{\3} (\Aee ) ) $, 
where $ \map{\varpi }{ I_{\3}^{\6} }{ \Theta  }$ is the map 
such that $ \varpi ((x_1,\ldots,x_n)) = \sum \delta _{x_i}$. 

By using \eqref{:49a} and \lref{l:42} for $ \3 $ 
it is not difficult to see there exists a constant 
$  \Ct \label{;44} $ independent of $ \e $ satisfying the following: 
\begin{align*}&
\mathbf{I}_{\e } \le \frac{ \cref{;44} }{|\log \e |}
.\end{align*}
This implies 
$ \lim_{\e  \to 0}  \mathcal{E}(\mathfrak{g}_{\e  } , \mathfrak{g}_{\e  } ) = 0 $.  
By \eqref{:48a} and \eqref{:48c} we have 
\begin{align}\notag &
(\mathfrak{g}_{\e  } ,\mathfrak{g}_{\e  } )_{\Lm } = \int _{\Ab } \mathfrak{g}_{\e  } ^2 d \mu \le n^2(n+1)^2 \, \mu (\Ab ) \to 0 \quad \text{ as }\e  \downarrow 0
.							\end{align}
Combining these with \eqref{:4a} we complete the proof of \lref{l:45}. 
\end{proof}

\noindent
{\em Proof of \tref{l:20}}.  
\tref{l:20} follows from \lref{l:43} and \lref{l:45} immediately.
\qed

\section{Proof of \pref{l:27}}\label{s:5} 

\begin{lem} \label{l:51}
Let $ \mu $ be a probability measure on $ (\Theta , \mathcal{B}(\Theta ) ) $ 
such that $ \mu (\{ \theta (\mathsf{E}) < \infty  \} ) = 1 $ and that 
density functions $ \{ \sigma ^n _{\mathsf{E} } \} $ on $ \mathsf{E} $ of 
$ \mu $ are continuous. 
Then $ (\mathcal{E}, \di ) $ is closable on $ \Lm $.  
\end{lem}
\begin{proof}
Let $ \Theta ^n = \{ \theta \in \Theta\, ;\, \theta (\mathsf{E} ) = n \}  $ 
and set 
$$ \mathcal{E}^n (\mathfrak{f}, \mathsf{g} ) = 
\sum _{k=1}^n \int _{\Theta ^k } \mathbb{D} [\mathfrak{f},\mathfrak{g}] d\mu .$$
By assumption $ \sum_{n=0}^{\infty } \mu (\Theta ^n) = 1$, from which we deduce 
$ (\mathcal{E}, \di ) $ is the increasing limit of $ \{ (\mathcal{E}^n, \di ) \} $. 
Since density functions are continuous, each $ (\mathcal{E}^n, \di ) $ 
is closable on $ \Lm $. So its increasing limit $ (\mathcal{E}, \di ) $ is also closable on $ \Lm $.   
\end{proof}

\begin{lem} \label{l:52}
Let $ \mu $ be a determinantal random point field on $ \mathsf{E}  $ 
with continuous kernel $ \mathsf{K}  $. 
Assume $ \mathsf{K}  $ is of trace class. 
Then their density functions $ \sigma ^n $ on $ \mathsf{E}  $ are continuous. 
\end{lem}
\begin{proof}
For the sake of simplicity we only prove the case $ K < 1 $, 
where $ K $ is the operator generated by the integral kernel $ \mathsf{K}  $. 
The general case is proved similarly by using a device in \cite[935 p.]{so-}.

Let $ \lambda _i $ denote the $ i $-th eigenvalue of $ K  $ and $ \varphi _i $ 
its normalized eigenfunction.  Then since $ K  $ is of trace class 
we have 
\begin{align} \label{:51}&
\mathsf{K}  (x,y) = \sum _{i=1}^{\infty} 
\lambda _i \varphi _i (x) \overline{\varphi _i (y)} 
.\end{align}

It is known that (see \cite[934 p.]{so-})
\begin{align}\label{:53}&
\sigma ^n (x_1,\ldots,x_n) = 
\det (\text{Id} - K  ) \cdot \det (L (x_i, x_j))_{1\le i,j \le n}
,\end{align}
where $ \det (\text{Id} - K  ) = \prod_{i=1}^{\infty} (1 - \lambda _i ) $ 
and 
\begin{align} \label{:54}&
L (x,y) = \sum_{i=1}^{\infty} \frac{\lambda _i}{1- \lambda _i} 
\varphi _i (x) \overline{\varphi _i (y)}
.
\end{align}

Since $ \mathsf{K} (x,y) $ is continuous, eigenfunctions 
$ \varphi _i (x) $ are also continuous. 
It is well known that the right hand side of \eqref{:51} converges uniformly. 
By $ 0 \le K < 1 $ we have $ 0 \le \lambda _i \le \lambda _1 <1 $.  
Collecting these implies the right hand side of \eqref{:54} converges uniformly. 
Hence $ L(x,y) $ is continuous in $ (x,y) $. 
This combined with \eqref{:53} completes the proof. 
\end{proof}

{\em Proof of \pref{l:27}}. 
Since $ K $ is of trace class, the associated determinantal random point field 
$ \mu $ satisfies $ \mu (\{ \theta (\mathsf{E}) < \infty  \} ) = 1 $. 
By \lref{l:52} 
we have density functions $ \sigma _{\mathsf{E} }^n $ are continuous. 
So \pref{l:27} follows from \lref{l:51}. 
\qed

\bigskip 

We now turn to the proof of \tref{l:26}. 
So as in the statement in \tref{l:26} let 
$ \mathsf{E} = \R  $ and 
$ \mathsf{K}(x,y) = \mathsf{m}(x) \mathsf{k}(x-y) \mathsf{m}(y)   $, where  
$ \map{\mathsf{k}}{\R }{\R } $ is a non-negative, continuous {\em even} function that is convex in $ [0, \infty) $ such that $ \mathsf{k} (0) \le 1 $, 
and $ \map{\mathsf{m}}{\R }{\R }   $ is nonnegative continuous and $ \int _{\R } \mathsf{m}(t) dt < \infty  $ and $ \mathsf{m} (x) \le 1  $ for all $ x  $ and 
$ 0 < \mathsf{m} (x) $ for some $ x $.  
We assume 
 $ \mathsf{k}  $ satisfies \eqref{:27b}.

\begin{lem} \label{l:53}
There exists an interval $ I $ in $ \mathsf{E}  $ such that 
\begin{align}\label{:55} &
\sigma _{I }^2 (x,x+t) \ge \cref{;51}  t ^{\alpha } 
\quad \text{ for all } |t| \le 1 \text{ and } x, x+t \in I 
,\end{align}
where $ \Ct \label{;51} $ is a positive constant and $ \sigma _{I }^2  $ 
is the 2-density function of $ \mu $ on $ I $. 
\end{lem}
\begin{proof}
By assumption  we see $ \inf _{ x \in I} \mathsf{m}(x) > 0 $ 
for some open bounded, nonempty interval $ I $ in $ \mathsf{E}  $. 
By \eqref{:4u} we have 
\begin{align}\label{:56}&
\sigma _{I }^2 (x,x+t) \ge \rho _2 (x, x+t) - 
\int _I \rho_3 (x, x+t, z) dz 
\end{align}
By \eqref{:23} and \eqref{:27b} there exist positive constants 
$ \Ct \label{;52} $ and $ \Ct \label{;53} $ such that 
\begin{align}\label{:57}&
\cref{;52}  t ^{\alpha } \le \rho _2 (x,x+t) 
&&\quad \text{ for all } |t| \le 1 \text{ and } x, x+t \in I 
\\ &\notag 
\rho _3 (x,x+t,z ) \le \cref{;53}  t ^{\alpha } 
&&\quad \text{ for all } |t| \le 1 \text{ and } x, x+t , z \in I 
.\end{align}
Hence by taking $ I $ so small we deduce \eqref{:55} 
from \eqref{:56} and \eqref{:57}. 
\end{proof}

{\em Proof of \tref{l:26}}. 
The closability follows from \pref{l:27}. 
So it only remains to prove \eqref{:27c}. 

Let $ (\mathcal{E}^2 , \dom ^2 ) $ and $ (\mathcal{E}, \dom ) $ 
denote closures of 
$ (\mathcal{E} ^2, \di   ) $ and 
$ (\mathcal{E} , \di ) $ 
on $ \Lm $, respectively.  Then 
\begin{align}\label{:5a}&
 (\mathcal{E} ^2 , \dom ^2 ) \le (\mathcal{E} , \dom )   
\end{align}
Let $ I $ be as in \lref{l:53}. 
Let $ \{ I_r \} _{r=1,\ldots}$ be an increasing sequence of 
 open intervals in $ \mathsf{E}  $ 
such that $ I_1 = I $ and $ \cup_r I _r = \mathsf{E} $.  
Let 
\begin{align}\label{:58}&
\mathcal{E}_r^2  (\mathfrak{f},\mathfrak{g}  )= \int _{\Theta^2 } 
\sum _{x_i \in I_r} \frac{1}{2} 
\frac{\partial f (\mathbf{x} )}{\partial x_i } \cdot 
\frac{\partial g (\mathbf{x} )}{\partial x_i }
d \mu
\end{align}
Here we set $ \mathbf{x}=(x_1,\ldots )  $, $ f $  and $ \mathfrak{f}  $ 
similarly as in \eqref{:25}. Then since density functions on $ I_r $ 
are continuous, we see $ (\mathcal{E}_r^2 , \di ) $ are closable on $ \Lm $.  
So we denote its closure by $ (\mathcal{E}_r^2  , \dom _r^2 ) $.  
It is clear that $ \{ (\mathcal{E}_r^2  , \dom _r^2 ) \}  $ is increasing 
in the sense that $ \dom _r^2 \supset \dom _{r+1}^2 $ and 
$ \mathcal{E}_r^2  (\mathfrak{f}, \mathfrak{f}  ) \le 
\mathcal{E}_{r+1}^2  (\mathfrak{f}, \mathfrak{f}  ) $ for all 
$ \mathfrak{f} \in \dom _{r+1} $.  
So we denote its limit by $ (\check{\mathcal{E} }^2 , \check{\dom}^2  ) $.  
It is known (\cite[Remark \thetag{3} after Theorem 3]{o.dfa}) that 
\begin{align}\label{:59}&
(\check{\mathcal{E}}^2 , \check{\mathcal{\dom} }^2  ) \le 
(\mathcal{E}^2, \dom ^2) 
.\end{align}

By \eqref{:5a}, \eqref{:59} and the definition of 
$ \{ (\mathcal{E}_r^2  , \dom _r^2 ) \}  $  we conclude 
$ (\mathcal{E}_1^2  , \dom _1^2  ) \le (\mathcal{E} , \dom )  $, 
which implies 
\begin{align}\label{:5b}&
\0 _1^2  \le \0 
,\end{align}
where $ \0 _1^2  $ and $ \0  $ denote capacities of 
$ (\mathcal{E}_1^2  , \dom _1^2  ) $ and $ (\mathcal{E} , \dom  ) $, respectively. 
Let $ \mathsf{B}=\Theta ^2 \cap 
\{ \theta (\{ x \} ) =2 \text{ for some } x \in I \}  $. 
Then by \eqref{:27a} and \eqref{:55} together with 
 a standard argument (see \cite[Example 2.2.4]{fot} for example)
we obtain 
\begin{align}\label{:5c}&
0 < \0 _1^2  (\mathsf{B} ) .
\end{align}Since $ \mathsf{B}  \subset \mathsf{A}  $, 
we deduce $ 0 < \0 (\mathsf{A} ) $ from \eqref{:5b} and \eqref{:5c}, 
which implies \eqref{:27c}.
\qed

\section{A construction of infinite volume dynamics}\label{s:6}
In this section we prove \tref{l:28}. 
We first prove the closability of pre-Dirichlet forms in finite volume.  
\begin{lem} \label{l:61}
Let $ I _r = (-r, r) \cap \mathsf{E}  $ and $ \sigma _r^n  $ denote the 
$ n $-density function on $ I _r $.   
Then $ \sigma _r^n  $ is continuous. 
\end{lem}
\begin{proof}
Let $ M = \sup _{x,y \in I _r} | \mathsf{K}(x,y)| $. 
Then $ M < \infty $ because $ \mathsf{K}  $ is continuous.  
Let $ \mathbf{x}_i = 
(\mathsf{K} (x_i, x_1), \mathsf{K}  (x_i, x_2),\ldots, 
\mathsf{K}  (x_i, x_n))  $ and 
$ \| \mathbf{x}_i \| $ denote its Euclidean norm. Then by \eqref{:23} we see 
\begin{align}\label{:61}&
|\rho _n|  \le  \prod _{i=1}^n \| \mathbf{x}_i \| \le \{ \sqrt{n} M  \}^n 
.\end{align}
By using Stirling's formula and \eqref{:61} 
we have for some positive constant $ \Ct \label{;61} $  
independent of $ k $ and $ M $ such that 
\begin{align}\label{:62}&
| \frac{(-1)^k}{k !}
\int _{I_r^{k }} \rho _{\6 + k} (x_1,\ldots , x_{\6+k}) 
dx_{\6+1}\cdots dx_{\6+k} |
\\\notag 
& \quad  \quad \quad \quad \quad \le 
\cref{;61}^k  k^{- k + 1/2} (n+k)^{(n+k)/2} M ^{n+k} 
.\end{align}
This implies for each $ n $ the series in the right hand side of \eqref{:4u} 
converges uniformly in $ (x_1,\ldots,x_n) $. 
So $ \sigma _r^n  $ is the limit of continuous functions in the uniform norm, which completes the proof. 
\end{proof}

\begin{lem} \label{l:62}
$ (\mathcal{E}, \dom _{\infty , r}  )  $ are closable on $ \Lm $. 
\end{lem}
\begin{proof}
Let $ I_r = \{ x \in \mathsf{E} ; |x| < r \}  $ and 
$ \Theta _r^n = \{ \theta (I_r) = n \}  $.  
Let 
$ \mathcal{E} _r^n (\mathfrak{f},\mathfrak{g}) 
= \int _{\Theta _r^n } \mathbb{D} [\mathfrak{f}, \mathfrak{g} ] d \mu $.  
Then it is enough to show that 
$ (\mathcal{E} _r^n , \dom _{\infty , r}) $ 
are closable on $ \Lm $ for all $ n $.  

Since $ \mathfrak{f}  $ is $ \sigma[\pi _r] $-measurable, we have  
($ \mathbf{x}=  (x_1,\ldots,x_n) $)
\begin{align*}&
\mathcal{E} _r^n (\mathfrak{f},\mathfrak{g}) = \frac{1}{n!}
\int _{I_r^n } 
\sum _{i=1}^n \frac{1}{2} 
\frac{\partial f _r^n (\mathbf{x} )}{\partial x_i } \cdot 
\frac{\partial g _r^n (\mathbf{x} )}{\partial x_i }
 \sigma _r^n (\mathbf{x} )d \mathbf{x} 
,\end{align*}
where $ f _r^n $ and $ g _r^n $ are defined similarly as after \eqref{:3.5}. 
Then since $ \sigma _r^n  $ is continuous, 
we see $ (\mathcal{E} _r^n , \dom _{\infty , r} ) $ is closable.  
\end{proof}

\begin{proof}[Proof of \tref{l:28}]
By \lref{l:62} we see the assumption \thetag{A.1$ ^* $} in \cite{o.dfa} 
is satisfied. \thetag{A.2} in \cite{o.dfa} is also satisfied by the construction of determinantal random point fields. So one can apply results 
in \cite{o.dfa} (Theorem 1, Corollary 1, Lemma 2.1 \thetag{3} in \cite{o.dfa}) to the present situation. 
Although in Theorem 1 in \cite{o.dfa} we treat $ (\mathcal{E}, \dom )  $, 
it is not difficult to see that the same conclusion also holds for 
$ (\Ereg , \Dreg ) $, which completes the proof. 
\end{proof}

\section{Gibbsian case}\label{s:7}
In this section we consider the case $ \mu $ is a canonical Gibbs measure 
with interaction potential $ \Phi $, whose $ n $-density functions for 
bounded sets are bounded, and 1-correlation function is locally integrable. 
If $ \Phi $ is super stable and regular in the sense of Ruelle, then 
probability measures satisfying these exist. 
In addition, it is known in \cite{o.dfa} that, if  $ \Phi $ is upper semi-continuous 
(or more generally $ \Phi $ is a measurable function dominated from above by 
a upper semi-continuous potential satisfying certain integrable conditions 
(see \cite{o.m})), 
then the form $ (\mathcal{E}, \dom ) $ on $ \Lm $ is closable. 
We remark these assumptions are quite mild.  
In \cite{o.dfa} and \cite{o.m} only {\em grand} 
canonical Gibbs measures with {\em pair} interaction potential are treated; it is easy to generalize the results in \cite{o.dfa} and \cite{o.m} 
to the present situation.

\begin{prop}\label{l:71}
Let $ \mu $ be as above. Assume $ d \ge 2 $. Then 
$ \0 (\mathsf{A} ) = 0 $ and no collision \eqref{:27} occurs. 
\end{prop}

\begin{proof}
The proof is quite similar to the one of \tref{l:20}. 
Let $ \mathbf{I}_{\e }  $ be as in \eqref{:4b}.  
It only remains to show $ \lim_{\e \to 0} \mathbf{I}_{\e }  = 0 $. 

We divide the case into two parts: \thetag{1} $ d = 2 $ and \thetag{2} $ 3 \le d $.  Assume \thetag{1}. We can prove $ \lim \mathbf{I}_{\e } =0 $ similarly as 
before. 
In the case of \thetag{2} the proof is more simple. 
Indeed, we change definitions of 
$ \Ab  $ 
in \eqref{:45} and $ h_{\e } $ in \eqref{:46} as follows:
$ \Ab = \mathsf{A}_{4 \e  }^{\mathbf{a} } (\7,\6,\8)  $  
\begin{align}\label{:71}&
\4 (t) = 
\begin{cases} 
2 								& ( |t| \le \e )
\\  -(2/\e) |t| + 4  		&( \e \le |t| \le 2\e   ) 
\\0 							& ( 2\e  \le |t| )
.\end{cases}
\end{align}
Then we can easily see $ \lim \mathbf{I}_{\e } =0 $. 
\end{proof}

\begin{rem}\label{r:71} \thetag{1} 
This result was announced and used in \cite[Lemma 1.4]{o.inv2}. 
Since this result was so different from other parts of the paper \cite{o.inv2}, 
we did not give a detail of the proof there. 
\\
\thetag{2} In \cite{r-s} a related result was obtained. 
In their frame work the choice of the domain of Dirichlet forms 
may be not same as ours. Indeed, their domains are 
smaller than or equal to ours 
(we do not know they are same or not). 
So one may deduce \pref{l:71} from their result. 
\end{rem}

\bigskip

\noindent 
Address: \\
Graduate School of Mathematics\\
Nagoya University \\
Chikusa-ku, Nagoya, 464-8602\\

\bigskip

\noindent 
Current Address (2015) \\
 Hirofumi Osada\\
 Tel  0081-92-802-4489 (voice)\\
 Graduate School of Mathematics,\\
 Kyushu University\\
 Fukuoka 819-0395, JAPAN\\
{\em \texttt{osada@math.kyushu-u.ac.jp}}

\noindent 
submit: February 9, 2003, revised: March 31, 2003

\noindent 
{\em Partially supported by Grant-in-Aid for Scientific Research (B) 11440029}

\end{document}